\documentclass[a4paper,10pt]{article}

\usepackage{graphicx}
\usepackage{amsfonts}
\usepackage{amsthm}
\usepackage{amsmath}
\usepackage{amssymb}
\usepackage{mathtools} 

\usepackage[top=1in, bottom=1in, left=1in, right=1in]{geometry}

\usepackage{graphicx}
\usepackage{epstopdf}
\usepackage{subfigure}
\usepackage[center,bf]{caption}

\usepackage{authblk}

\theoremstyle{plain}
\newtheorem{theorem}{Theorem}[section]
\newtheorem{corollary}[theorem]{Corollary}
\newtheorem{lemma}[theorem]{Lemma}

\theoremstyle{definition}
\newtheorem{definition}[theorem]{Definition}

\theoremstyle{remark}

\DeclareSymbolFont{pxfontssymbolsC}{U}{pxsyc}{m}{n}
\DeclareMathSymbol{\coloneqq}{\mathrel}{pxfontssymbolsC}{66}

\begin{document}

\title{An elliptic regularity theorem for fractional partial differential operators}

\author[1]{Arran Fernandez \thanks{Email: \texttt{af454@cam.ac.uk}}}
\affil[1]{{\small Department of Applied Mathematics \& Theoretical Physics, University of Cambridge, Cambridge, CB3 0WA, United Kingdom}}

\date{}


\maketitle


\begin{abstract}\noindent We present and prove a version of the elliptic regularity theorem for partial differential equations involving fractional Riemann--Liouville derivatives. In this case, regularity is defined in terms of Sobolev spaces $H^s(X)$: if the forcing of a linear elliptic fractional PDE is in one Sobolev space, then the solution is in the Sobolev space of increased order corresponding to the order of the derivatives. We also mention a few applications and potential extensions of this result.
\end{abstract}

\section{Introduction}
\label{sec:intro}
In fractional calculus, the orders of differentiation and integration are extended beyond the integer domain to the real line and even the complex plane. This field of study has a long history, having been considered by Leibniz, Riemann, and Hardy among others \cite{miller}. It also has a wide variety of applications, including in bioengineering \cite{glockle,magin}, chaos theory \cite{zaslavsky}, drug transport \cite{dokoumetzidis,petras,sopasakis}, epidemiology \cite{carvalho}, geohydrology \cite{atangana2}, random walks \cite{zaburdaev}, thermodynamics \cite{vazquez}, and viscoelasticity \cite{koeller}. Many of these cited papers are from the last few years, indicating the importance and relevancy of fractional calculus in modern science.

Fractional derivatives and integrals can be defined in several different ways, not all of which agree with each other, and thus it must always be clear which definition is being used. In fact, new models of fractional calculus are being developed all the time: see for example \cite{caputo,atangana,jarad} for some fractional models developed only in the last few years. In this paper, however, we shall always use the classical Riemann--Liouville formula (Definitions \ref{RLintegral} and \ref{RLderivative}) unless otherwise stated.

\begin{definition}[Riemann--Liouville fractional integral]
\label{RLintegral}
Let $x$ and $\nu$ be complex variables, and $c$ be a constant in the extended complex plane (usually taken to be either $0$ or negative real infinity). For $\mathrm{Re}(\nu)<0$, the $\nu$th derivative, or $(-\nu)$th integral, of a function $f$ is \[D_{c+}^{\nu}f(x)\coloneqq\frac{1}{\Gamma(-\nu)}\int_c^x(x-y)^{-\nu-1}f(y)\,\mathrm{d}y,\] provided that this expression is well-defined. (If $c=-\infty$, the operator is denoted by simply $D_+^{\nu}$ instead of $D_{c+}^{\nu}$.)
\end{definition}

Since $x$, $\nu$, and $c$ are defined to be in the complex plane, we must consider the issue of which branch to use when defining the function $(x-y)^{-\nu-1}$ and which contour from $c$ to $x$ to use for the integration. Clearly $\arg(x-y)$ can be fixed to be always equal to $\arg(x-c)$, i.e. by taking the contour of integration to be the straight line-segment $[c,x]$. And the choice of range for $\arg(x-c)$ usually depends on context: the essential properties of Riemann--Liouville integrals remain unchanged whether $\arg(x-c)$ is assumed to be in $[0,2\pi)$, $(-\pi,\pi]$, or any other range. Here we shall follow \cite[\textsection 22]{samko} in using $\arg(x-c)\in[0,2\pi)$, because we shall usually be assuming $c=-\infty$ and $x\in\mathbb{R}$, in which case $\arg(x-c)=0$ is the obvious choice to make.

\begin{definition}[Riemann--Liouville fractional derivative]
\label{RLderivative}
Let $x,\nu,c$ be as in Definition \ref{RLintegral} except with $\mathrm{Re}(\nu)\geq0$. The $\nu$th derivative of a function $f$ is \[D_{c+}^{\nu}f(x)\coloneqq\tfrac{d^n}{dx^n}\big(D_{c+}^{\nu-n}f(x)\big),\] where $n\coloneqq\lfloor\mathrm{Re}(\nu)\rfloor+1$, provided that this expression is well-defined. (Again, if $c=-\infty$, the operator is denoted by simply $D_+^{\nu}$ instead of $D_{c+}^{\nu}$.)
\end{definition}

For functions $f$ such that $D_{c+}^{\nu}f(x)$ is analytic in $\nu$, Definition \ref{RLderivative} is the analytic continuation in $\nu$ of Definition \ref{RLintegral}. This provides some motivation for why this formula should be used.

When the order of differentiation and integration becomes continuous, the term \textit{differintegration} is often used to cover both. When the order of differintegration lies in the complex plane, its real part is what defines the difference between differentiation and integration.

In the case where $f$ is holomorphic, the following definition (Definition \ref{Cauchy}) can be more useful for applications in complex analysis. It is equivalent to the Riemann--Liouville definition wherever both are defined, as proved in \cite[Chapter 3]{oldham}.

\begin{definition}[Cauchy fractional differintegral]
\label{Cauchy}
Let $x$ and $\nu$ be complex variables, and $c$ be a constant in the extended complex plane. For $\nu\in\mathbb{C}\backslash\mathbb{Z}^-$, the $\nu$th derivative of a function $f$ analytic in a neighbourhood of the line segment $[c,x]$ is \[D_{c+}^{\nu}f(x)\coloneqq\frac{\Gamma(\nu+1)}{2\pi i}\int_{\mathcal{H}}(y-x)^{-\nu-1}f(y)\,\mathrm{d}y,\] provided that this expression is well-defined, where $\mathcal{H}$ is a finite Hankel-type contour with both ends at $c$ and circling once counter-clockwise around $x$.
\end{definition}

Note that Definition \ref{RLintegral} is the natural generalisation of the Cauchy formula for repeated real integrals (see \cite[Chapter II]{miller}), while Definition \ref{Cauchy} is similarly the natural generalisation of Cauchy's integral formula from complex analysis.

Since the Riemann--Liouville fractional derivative is defined using ordinary derivatives of fractional integrals, one might wonder what would happen if the order of these operations were reversed. Using fractional integrals of ordinary derivatives instead, we obtain a different definition of fractional differentiation, this one due to Caputo.

\begin{definition}[Caputo fractional derivative]
\label{Caputo}
Let $x,\nu,c$ be as in Definition \ref{RLintegral} except with $\mathrm{Re}(\nu)\geq0$. The $\nu$th derivative of a function $f$ is \[D_{c+}^{\nu}f(x)\coloneqq D_{c+}^{\nu-n}\Big(\tfrac{d^nf}{dx^n}\Big),\] where $n\coloneqq\lfloor\mathrm{Re}(\nu)\rfloor+1$, provided that this expression is well-defined.
\end{definition}

Fractional \textit{integrals} in the Caputo context are exactly Riemann--Liouville integrals, so a new definition is not needed for them. Lemma \ref{ComposeDerivatives} below shows that the Riemann--Liouville and Caputo fractional derivatives (Definitions \ref{RLderivative} and \ref{Caputo}) are not equivalent in general.

The constant $c$ used in the above definitions can be thought of as a constant of integration. However, in the fractional context it appears in the formulae for derivatives as well as those for integrals. It is almost always assumed to be either $0$ or $-\infty$.

Some standard properties of integer-order differintegrals extend to the fractional case: for instance, $D_{c+}^{\nu}$ is still a linear operator for any $\nu$ and $c$. But other standard theorems of calculus no longer hold in the fractional case, or hold in a more complicated way. For instance, the fractional derivative of a fractional derivative is not always a fractional derivative; composition of fractional differintegrals is governed by the equations in Lemmas \ref{ComposeIntegrals} and \ref{ComposeDerivatives}.

\begin{lemma}[Composition of fractional integrals]
\label{ComposeIntegrals}
For any $x,\mu,\nu\in\mathbb{C}$ with $\mathrm{Re}(\mu)<0$ and any function $f$ continuous in a neighbourhood of $c$, the identity $D_{c+}^{\nu}\big(D_{c+}^{\mu}f(x)\big)=D_{c+}^{\mu+\nu}f(x)$ holds provided these differintegrals exist.
\end{lemma}

\begin{proof}
This is a simple exercise in manipulation of double integrals, and may be found in \cite[Chapter 2.3.2]{podlubny}.
\end{proof}

\begin{lemma}
\label{ComposeDerivatives}
If $n\in\mathbb{N}$ and $f$ is a $C^n$ function such that one of $D_{c+}^n\big(D_{c+}^{\mu}f(x)\big),\,D_{c+}^{n+\mu}f(x),\,D_{c+}^{\mu}\big(D_{c+}^nf(x)\big)$ exists, then all three exist and \[D_{c+}^n\big(D_{c+}^{\mu}f(x)\big)=D_{c+}^{n+\mu}f(x)=
D_{c+}^{\mu}\big(D_{c+}^nf(x)\big)+\sum_{k=1}^n\frac{(x-c)^{-\mu-k}}{\Gamma(-\mu-k+1)}f^{(n-k)}(c).\]
\end{lemma}

\begin{proof}
The first identity follows directly from Definition \ref{RLderivative} for Riemann--Liouville fractional derivatives. For the second, use induction on $n$, starting with the $\mathrm{Re}(\mu)<0$ case and using integration by parts, then proving the $\mathrm{Re}(\mu)\geq0$ case by performing ordinary differentiation on the previous case. A more detailed proof can be found in \cite[Chapter III]{miller}.
\end{proof}

Note that when $c$ is infinite and $f$ has sufficient decay conditions, the series term disappears. In this case, the Riemann--Liouville and Caputo fractional derivatives (Definitions \ref{RLderivative} and \ref{Caputo}) are equivalent. This fact will be used in Lemma \ref{SchwartzConvolution} below.

Another definition of fractional calculus involves generalising the relationship given by the Fourier transform between differentiation and multiplication by power functions. In fact, Lemma \ref{Fourier} shows that this model, commonly used in applications involving partial differential equations, is equivalent to the Riemann--Liouville model with $c=-\infty$. Similarly, Lemma \ref{Laplace} shows that the corresponding definition with Laplace transforms instead of Fourier is equivalent to the Riemann--Liouville model with $c=0$.

\begin{lemma}[Fourier transforms of fractional differintegrals]
\label{Fourier}
If $f(x)$ is a function with well-defined Fourier transform $\hat{f}(\lambda)$ and $\nu\in\mathbb{C}$ is such that $D_+^{\nu}f(x)$ is well-defined, then the Fourier transform of $D_+^{\nu}f(x)$ is $(-i\lambda)^{\nu}\hat{f}(\lambda)$.
\end{lemma}

\begin{proof}
If $\mathrm{Re}(\nu)<0$, then Definition \ref{RLintegral} can be rewritten as a convolution: $D_+^{\nu}f=f\ast\Phi$ where $\Phi(x)=\frac{x^{-\nu-1}}{\Gamma(-\nu)}$ when $x>0$, $\Phi(x)=0$ otherwise. Convolutions transform to products under the Fourier transform, so the result follows.

If $\mathrm{Re}(\nu)\geq0$, the result follows from the fractional integral case (proved above) and the $\nu\in\mathbb{N}$ case (which is standard).
\end{proof}

\begin{lemma}[Laplace transforms of fractional integrals]
\label{Laplace}
If $f(x)$ is a function with well-defined Laplace transform $\tilde{f}(\lambda)$, and $\nu\in\mathbb{C}$ with $\mathrm{Re}(\nu)<0$ is such that $D^{\nu}_{0+}f(x)$ is well-defined, then the Laplace transform of $D^{\nu}_{0+}f(x)$ is $(-i\lambda)^{\nu}\tilde{f}(\lambda)$.
\end{lemma}

\begin{proof}
As for Lemma \ref{Fourier}. See also \cite[Chapter III]{miller}.
\end{proof}

The corresponding result for Laplace transforms of fractional \textit{derivatives} is more complicated, because of the initial values terms arising. It may be found in \cite[Chapter IV]{miller}.

Finally, we demonstrate one way, due to Osler, in which the product rule -- another basic result of classical calculus -- can be extended to Riemann--Liouville fractional calculus.

\begin{lemma}[The fractional product rule]
\label{Osler}
Let $u$ and $v$ be complex functions such that $u(x)$, $v(x)$, and $u(x)v(x)$ are all functions of the form $x^{\lambda}\eta(x)$ with $\mathrm{Re}(\lambda)>-1$ and $\eta$ analytic on a domain $R\subset\mathbb{C}$. Then for any distinct $x,c\in R$ and any $\nu\in\mathbb{C}$, we have \[D_{c+}^{\nu}\big(u(x)v(x)\big)=\sum_{n=0}^{\infty}\tbinom{\nu}{n}D_{c+}^{\nu-n}u(x)D_{c+}^{n}v(x),\] where all differintegrals are defined using the Cauchy formula.
\end{lemma}

\begin{proof}
See \cite{osler3}.
\end{proof}

Partial differential equations (PDEs) of fractional order have also become an important area of study, with entire textbooks written about them and their applications \cite{kilbas,podlubny}. A huge variety of methods have been devised for solving them, including by extending known results of classical calculus: see for example \cite{podlubny2,yang,bin,baleanu} among many others. The non-locality of fractional derivatives lends them utility in many real-life problems, e.g. in control theory, dynamical systems, and elasticity theory \cite{baleanu2,luchko,tarasov}.

The \textbf{elliptic regularity theorem} is an important result in the theory of partial differential equations. In its most general form, it says that for any PDE satisfying certain conditions, there are regularity properties of the solution function which depend naturally on the regularity properties of the forcing function. This is useful in cases where the solution function cannot be constructed explicitly: more information about its essential properties is the next best thing to an analytic solution.

Here we shall focus on the version of the elliptic regularity theorem given in Theorem \ref{ERT}, in which the PDE must be linear and elliptic with constant coefficients, and `regularity' is defined in terms of Sobolev spaces.

\begin{definition}
\label{Sobolev}
For any real number $s$ and any natural number $n$, the $s$th Sobolev space on $\mathbb{R}^n$ is defined to be \[H^s(\mathbb{R}^n)\coloneqq\{u\in\mathcal{S}'(\mathbb{R}^n):\hat{u}\in L^2_{loc}(\mathbb{R}^n),||u||_{H^s}<\infty\},\] where the Sobolev norm $||\cdot||_{H^s}$ is defined by \[||u||_{H^s}\coloneqq\left(\int_{\mathbb{R}^n}|\hat{u}(\lambda)|^2\left(1+|\lambda|^2\right)^s\,\mathrm{d}\lambda\right)^{1/2}.\] For a general domain $X\subset\mathbb{R}^n$, the $s$th Sobolev space on $X$ is defined to be \[H^s_{loc}(X)\coloneqq\{u\in\mathcal{D}'(X):u\phi\in H^s(\mathbb{R}^n)\,\text{ for all }\,\phi\in\mathcal{D}(X)\}.\]
\end{definition}

\begin{theorem}[Elliptic regularity theorem]
\label{ERT}
Let $P(D)$ be an elliptic partial differential operator given by a complex $n$-variable $N$th-order polynomial $P$ applied to the differential operator $D\coloneqq-i\frac{\partial}{\partial x}$ where $x$ is a variable in $\mathbb{R}^n$. If $X$ is a domain in $\mathbb{R}^n$ and $u,f\in\mathcal{D}'(X)$ satisfy $P(D)u=f$, then \[f\in H^s_{loc}(X)\Rightarrow u\in H^{s+N}_{loc}(X).\]
\end{theorem}

\begin{proof}
See \cite[Chapter 9]{folland2}.
\end{proof}

Related, more general, results are already known from the theory of pseudodifferential operators; see e.g. \cite[Theorem 7.13]{abels} for an example of an elliptic regularity theorem in this setting. However, it is not necessary to introduce the full machinery of pseudodifferential operators -- with associated stronger conditions on the forcing and solution functions -- in order to obtain a useful analogue of Theorem \ref{ERT} for fractional differential equations.

The structure of this paper is as follows. In Section \ref{sec:main}, the bootstrapping proof used in \cite{folland2} to prove Theorem \ref{ERT} is adapted, with some modifications and extra lemmas, to prove an elegant analogous result in the Riemann--Liouville fractional model. The place where most new work was needed was in the proof of Lemmas \ref{SobolevConvolution} and \ref{SchwartzConvolution}; the final result is Theorem \ref{FERT}. In Section \ref{sec:concl}, we consider applications and potential extensions of our work here.

\section{The main result}
\label{sec:main}
Let $x\in\mathbb{R}^n$ be an $n$-dimensional variable, and let $D$ denote the modified $n$-dimensional differential operator $-iD_+$ where $D_+$ is the vector operation of differentiation with respect to $x$ defined in Definitions \ref{RLintegral}--\ref{RLderivative}. In other words, the differential operator $D^{\alpha}$ is defined by \[D^{\alpha}f(x)=e^{-i\pi\alpha/2}D_{+}^{\alpha}f(x)\]We use the constant of differintegration $c=-\infty$ so that we can make use of Fourier transforms in the proof (by Lemma \ref{Fourier}), and also so that the Riemann--Liouville and Caputo fractional derivatives are equal (by the discussion following Lemma \ref{ComposeDerivatives}), which is required at a certain stage in the proof.

Let $P$ be a finite linear combination of power functions, i.e. \[P(\lambda)=\sum_{\alpha}c_{\alpha}\lambda^{\alpha},\] where $\alpha$ is a fractional multi-index in $(\mathbb{R}^+_0)^n$ and the sum is finite. This defines a fractional differential operator $P(D)$, where all powers of $D$ are defined using the Riemann--Liouville formula (Definition \ref{RLderivative}) with $c=-\infty$. The fractional partial differential equation we shall be considering is of the form \[P(D)u=f.\]

\begin{definition}
\label{order}
The \textbf{order} $\nu$ of the operator $P(D)$ defined above is the maximal $|\alpha|$ such that $c_{\alpha}\neq0$. Note that $\nu$ is not necessarily an integer, and that since $P$ is a finite sum, there exists $\epsilon>0$ such that $|\alpha|\leq\nu-\epsilon$ for every $\alpha$ such that $c_{\alpha}\neq0$ and $|\alpha|<\nu$.
\end{definition}

\begin{definition}
\label{elliptic}
The \textbf{principal symbol} of $P(D)$ is defined to be the function $\sigma_P(\lambda)=\sum_{|\alpha|=\nu}c_{\alpha}\lambda^{\alpha}$. The operator $P(D)$ is said to be \textbf{elliptic} if $\sigma_P(\lambda)\neq0$ for all nonzero $\lambda\in\mathbb{R}^n$.
\end{definition}

\begin{lemma}
\label{bounds}
If $P(D)$ is a $\nu$th-order elliptic fractional partial differential operator as above, then there exist positive real constants $C,R$ such that for any $\lambda\in\mathbb{C}^n$ with $|\lambda|>R$, the function $P$ satisfies $|P(\lambda)|\geq C(1+|\lambda|^2)^{\nu/2}$.
\end{lemma}

\begin{proof}
First consider the non-fractional case, i.e. where $P$ is a polynomial. Here $|\sigma_P|$ is a continuous positive function on the compact domain $|\lambda|=1$, so it has a positive lower bound on this domain. In other words, $|\sigma_P(\lambda)|\gg1$ when $|\lambda|=1$, which implies $|\sigma_P(\lambda)|\gg|\lambda|^{\nu}$ for all $\lambda$. By the triangle inequality, this implies 

\begin{equation}
|P(\lambda)|\gg\Big(1-\frac{|P(\lambda)-\sigma_P(\lambda)|}{|\lambda|^{\nu}}\Big)|\lambda|^{\nu}.
\end{equation}

\noindent Since $P(\lambda)-\sigma_P(\lambda)$ is a polynomial of order less than $\nu$, the ratio term is $\ll1$ when $|\lambda|$ is sufficiently large. So for large $|\lambda|$ we have $|P(\lambda)|\gg|\lambda|^{\nu}\gg(1+|\lambda|^2)^{\nu/2}$ as required.

The above proof relies on the continuity of the function $\sigma_P(\lambda)$, which is not true in general since $\lambda^{\alpha}$ has a branch cut in the complex $\lambda$-plane when $\alpha$ is not an integer. But $\sigma_P(\lambda)$ can be approximated arbitrarily closely by a sum of \textit{rational} powers of $\lambda$, i.e. a polynomial of order around $m\nu$ in $\lambda^{1/m}$ for some large natural number $m$. Call this function $\tilde{\sigma}_P(\lambda)$; the above proof shows that $|\tilde{\sigma}_P(\lambda)|\gg1$ when $|\lambda^{1/m}|=1$, i.e. when $|\lambda|=1$. Now by letting the exponents in $\tilde{\sigma}_P$ tend to those in $\sigma_P$, we find $|\sigma_P(\lambda)|\gg1$ when $|\lambda|=1$, as before. Again this gives equation (1).

Because of the finite bound $\epsilon$ mentioned in Definition \ref{order}, the ratio term $\frac{|P(\lambda)-\sigma_P(\lambda)|}{|\lambda|^{\nu}}$ is still $\ll1$ for sufficiently large $\lambda$, and the result follows.
\end{proof}

\begin{lemma}[Existence of parametrices]
\label{parametrices}
If P(D) is an elliptic fractional partial differential operator as above, then it has a \textit{parametrix}, i.e. $E\in\mathcal{D}'(\mathbb{R}^n)$ such that $P(D)E=\delta_0+\omega$ for some $\omega\in\mathcal{E}(\mathbb{R}^n)$, and the parametrix $E$ is in $\mathcal{S}'(\mathbb{R}^n)$ and also in $C^{\infty}(\mathbb{R}^n\backslash\{0\})$.
\end{lemma}

\begin{proof}
Fix a test function $\chi\in\mathcal{D}(\mathbb{R}^n)$ which is identically $1$ on the domain $|\lambda|\leq R$ and identically $0$ on the domain $|\lambda|>R+1$, where $R$ is as in Lemma \ref{bounds}. Let \[\hat{E}(\lambda):=\frac{1-\chi(\lambda)}{P(\lambda)}.\] This is well-defined because $1-\chi$ is zero at all zeros of $P$, and it is bounded by Lemma \ref{bounds}. By definition of $P$, we therefore have the leftmost of the following inclusions, leading to the rightmost: \[\hat{E}\in\mathcal{E}'(\mathbb{R}^n)\Rightarrow
\hat{E}\in\mathcal{S}'(\mathbb{R}^n)\Rightarrow
E\in\mathcal{S}'(\mathbb{R}^n)\Rightarrow E\in\mathcal{D}'(\mathbb{R}^n),\] where $E$ is the inverse Fourier transform of $\hat{E}$. Similarly, \[\chi\in\mathcal{D}(\mathbb{R}^n)\Rightarrow\chi\in\mathcal{S}(\mathbb{R}^n)\Rightarrow\omega\in\mathcal{S}(\mathbb{R}^n)\Rightarrow\omega\in\mathcal{E}(\mathbb{R}^n),\] where $\omega$ is the inverse Fourier transform of $-\chi$. Finally, \[P(\lambda)\hat{E}(\lambda)=1-\chi(\lambda)\Rightarrow P(D)E=\delta_0+\omega,\] so $E$ is a parametrix of $P(D)$.

On the domain $|\lambda|>R+1$, we have \[\Big|\widehat{D^{\alpha}(x^{\beta}E)}(\lambda)\Big|=\Big|\lambda^{\alpha}D^{\beta}E(\lambda)\Big|=\Big|\lambda^{\alpha}D^{\beta}\big(P(\lambda)^{-1}\big)\Big|\ll\big|\lambda\big|^{|\alpha|-|\beta|-\nu}\] for any multi-indices $\alpha,\beta$. So for all $\alpha,\beta$ with $|\beta|$ sufficiently large, the function $\widehat{D^{\alpha}(x^{\beta}E)}$ is in $L^1(\mathbb{R}^n)$, which means its inverse Fourier transform $D^{\alpha}(x^{\beta}E)$ is in $C(\mathbb{R}^n)$. So $E$ is in $C^{\infty}(\mathbb{R}^n\backslash\{0\})$. And the fact that $E\in\mathcal{S}'(\mathbb{R}^n)$ was already established above.
\end{proof}

\begin{lemma}
\label{SobolevConvolution}
If $\phi\in\mathcal{D}(\mathbb{R}^n)$ and $u\in H^t(\mathbb{R}^n)$ for some $n\in\mathbb{N},t\in\mathbb{R}$, then $[D^{\alpha},\phi](u)\in H^{t-|\alpha|+1}(\mathbb{R}^n)$ for any $\alpha\in\mathbb{C}^n$, where $[,]$ denotes a commutator.
\end{lemma}

\begin{proof}
Note that when $\alpha$ is an ordinary multi-index in $(\mathbb{Z}^+_0)^n$, this result is straightforwardly proved using the product rule: the operator $[D^{\alpha},\phi]$ is just an $(|\alpha|-1)$th-order differential operator. In the general case, however, we need to use infinite series and some more complicated estimates. It may appear that Osler's generalisation of the product rule (Lemma \ref{Osler}) is applicable, but analyticity is out of the question when we are dealing with test functions $\phi\in\mathcal{D}(\mathbb{R}^n)$.

The property of a function $f$ being in a Sobolev space $H^s(\mathbb{R}^n)$ depends only on the large-$\lambda$ behaviour of the Fourier transform $\hat{f}(\lambda)$, so it will suffice to prove that the Fourier transform of $[D^{\alpha},\phi](u)$ behaves like the Fourier transform of a function in $H^{t-|\alpha|+1}(\mathbb{R}^n)$ when $|\lambda|$ has some fixed lower bound.

Firstly, we rewrite the expression as follows:
\begin{align*}
\widehat{[D^{\alpha},\phi](u)}(\lambda)&=\widehat{D^{\alpha}(\phi u)}(\lambda)-\widehat{(\phi D^{\alpha}u)}(\lambda)=\lambda^{\alpha}\hat{\phi}(\lambda)\ast\hat{u}(\lambda)-\hat{\phi}(\lambda)\ast(\lambda^{\alpha}\hat{u}(\lambda)) \\
&=\lambda^{\alpha}\int_{\mathbb{R}^n}\hat{\phi}(\mu)\hat{u}(\lambda-\mu)\,\mathrm{d}\mu\,\,-\int_{\mathbb{R}^n}\hat{\phi}(\mu)(\lambda-\mu)^{\alpha}\hat{u}(\lambda-\mu)\,\mathrm{d}\mu \\
&=I_1(\lambda)+I_2(\lambda),
\end{align*}
where the two integral expressions $I_1,I_2$ are defined by
\begin{align*}
I_1(\lambda)&\coloneqq\lambda^{\alpha}\int_{|\mu|\leq\frac{1}{2}|\lambda|}\hat{\phi}(\mu)\Big(1-\big(1-\tfrac{\mu}{\lambda}\big)^{\alpha}\Big)\hat{u}(\lambda-\mu)\,\mathrm{d}\mu; \\
I_2(\lambda)&\coloneqq\int_{|\mu|>\frac{1}{2}|\lambda|}\hat{\phi}(\mu)\Big(\lambda^{\alpha}-(\lambda-\mu)^{\alpha}\Big)\hat{u}(\lambda-\mu)\,\mathrm{d}\mu.
\end{align*}
We shall evaluate $I_1$ and $I_2$ separately and prove bounds to establish that each of them is the Fourier transform of a function in $H^{t-|\alpha|+1}(\mathbb{R}^n)$, which will suffice to prove the lemma.

Firstly,
\begin{align*}
I_1(\lambda)&=\lambda^{\alpha}\int_{|\mu|\leq\frac{1}{2}|\lambda|}\hat{\phi}(\mu)\left[\sum_{m=1}^{\infty}\tbinom{\alpha}{m}(\tfrac{\mu}{\lambda})^m\right]\hat{u}(\lambda-\mu)\,\mathrm{d}\mu \\
&=\lambda^{\alpha}\int_{|\mu|\leq\frac{1}{2}|\lambda|}\hat{\phi}(\mu)\Big[\tbinom{\alpha}{1}\tfrac{\mu}{\lambda}+o\left(\tfrac{\mu}{\lambda}\right)\Big]\hat{u}(\lambda-\mu)\,\mathrm{d}\mu \\
&\sim\alpha\lambda^{\alpha-1}\int_{|\mu|\leq\frac{1}{2}|\lambda|}\mu\hat{\phi}(\mu)\hat{u}(\lambda-\mu)\,\mathrm{d}\mu \\
&\ll\alpha\lambda^{\alpha-1}\widehat{\phi'}(\lambda)\ast \hat{u}(\lambda) \\
&=\alpha\widehat{D^{\alpha-1}(\phi'u)}. 
\end{align*}
Since $\phi'\in\mathcal{D}(\mathbb{R}^n)$, we have $\phi'u\in H^t(\mathbb{R}^n)$. By Lemma \ref{Fourier}, this means the above expression is the Fourier transform of a function in $H^{t-|\alpha|+1}(\mathbb{R}^n)$, as required.

Now consider $I_2$. By the Paley-Wiener-Schwartz theorem (see \cite[Chapter 1]{hormander}), the function $\hat{\phi}$ is entire and satisfies an inequality of the form $|\hat{\phi}(\lambda)|\ll_{{}_N}(1+|\lambda|)^{-N}$ for $N\in\mathbb{N}$, $\lambda\in\mathbb{R}^n$, where the subscript means the multiplicative constant depends on $N$. So
\begin{align*}
I_2&=\int_{|\mu|>\frac{1}{2}|\lambda|}\hat{\phi}(\mu)\Big(\lambda^{\alpha}-(\lambda-\mu)^{\alpha}\Big)\hat{u}(\lambda-\mu)\,\mathrm{d}\mu \\
&\ll_{{}_N}\int_{|\mu|>\frac{1}{2}|\lambda|}(1+|\mu|)^{-N-|\alpha|}\big(|2\mu|^{|\alpha|}+|3\mu|^{|\alpha|}\big)|\hat{u}(\lambda-\mu)|\,\mathrm{d}\mu \\
&\ll\int_{|\mu|>\frac{1}{2}|\lambda|}(1+|\mu|)^{-N}|\hat{u}(\lambda-\mu)|\,\mathrm{d}\mu \\
&\ll \big((1+|\bullet|)^{-N}\ast|\hat{u}|\big)(\lambda).
\end{align*}
Since $u$ is in $H^t(\mathbb{R}^n)$ and $N$ can be arbitrarily large, this final expression must be the Fourier transform of a function in $H^{t+K}(\mathbb{R}^n)$ for arbitrarily large $K$. And $H^a\subset H^b$ for $a>b$, so $I_2$ is the Fourier transform of a function in $H^{t-|\alpha|+1}(\mathbb{R}^n)$, as required.
\end{proof}

\begin{lemma}
\label{SchwartzConvolution}
If $f$ and $g$ are functions, at least one of which is a Schwartz function, and $\nu\in\mathbb{C}$ is such that $D_+^{\nu}f$ and $D_+^{\nu}g$ are well-defined, then $D_+^{\nu}f\ast g=f\ast D_+^{\nu}g$, where $\ast$ denotes convolution.
\end{lemma}

\begin{proof}
When $\mathrm{Re}(\nu)<0$, writing $D_+^{\nu}f=f\ast\Phi$ as in Lemma \ref{Fourier} and using the associativity of convolution gives \[D_+^{\nu}f\ast g=(f\ast\Phi)\ast g=f\ast(\Phi\ast g)=f\ast(g\ast\Phi)=f\ast D_+^{\nu}g.\]

When $\mathrm{Re}(\nu)\geq0$ and $n$ is defined as in Definition \ref{Cauchy}, assuming without loss of generality that $g$ is a Schwartz function, using Definition \ref{RLderivative} and the above result gives \[D_+^{\nu}f\ast g=\Big(\tfrac{d^n}{dx^n}\big(D_+^{\nu-n}f\big)\Big)\ast g=D_+^{\nu-n}f\ast\tfrac{d^ng}{dx^n}=f\ast D_+^{\nu-n}\big(\tfrac{d^ng}{dx^n}\big).\] The final expression on the right-hand side is a Caputo derivative and not a Riemann--Liouville derivative of $g$. However, since $g$ is a Schwartz function, its Caputo and Riemann--Liouville derivatives are identical (by the discussion following Lemma \ref{ComposeDerivatives}), and the result follows.
\end{proof}

\begin{theorem}[Fractional elliptic regularity theorem]
\label{FERT}
If $P(D)$ is a $\nu$th-order elliptic fractional partial differential operator as above and $X$ is a domain in $\mathbb{R}^n$ and $u,f\in\mathcal{D}'(X)$ satisfy $P(D)u=f$, then \[f\in H^s_{loc}(X)\Rightarrow u\in H^{s+\nu}_{loc}(X).\]
\end{theorem}

\begin{proof}
First assume $X=\mathbb{R}^n$ and $u$ is compactly supported (i.e. in $\mathcal{E}'(\mathbb{R}^n)$). By Lemma \ref{parametrices}, $P(D)$ has a parametrix $E$ and (using Lemma \ref{SchwartzConvolution}) \[u=\delta_0\ast u=(P(D)E)\ast u-\omega\ast u=E\ast(P(D)u)-\omega\ast u=E\ast f-\omega\ast u.\] Since $u$ has compact support, $\omega\ast u$ is a Schwartz function, so it will be enough to prove $E\ast f\in H^{s+\nu}(\mathbb{R}^n)$. If $|\lambda|>R+1$, then by Lemma \ref{bounds} and the definition of $\hat{E}$, \[\Big|\widehat{E\ast f}(\lambda)\Big|=\Big|\tfrac{\hat{f}(\lambda)}{P(\lambda)}\Big|\ll(1+|\lambda|^2)^{-\nu/2}\hat{f}(\lambda).\] And $f\in H^s(\mathbb{R}^n)$, so $E\ast f\in H^{s+\nu}(\mathbb{R}^n)$ as required.

To prove the general case, we shall use a bootstrapping argument. First of all, let us note that it makes sense to define fractional derivatives of functions in $\mathcal{D}'(X)$ even when $X$ does not extend to negative infinity: the integrals from $-\infty$ to $x$ required by Definition \ref{RLintegral} are simply taken to be zero outside of $X$. In other words, the arbitrary test function $\phi\in\mathcal{D}(X)$ is extended to a function on all of $\mathbb{R}^n$ which is supported on $X$.

Fix $\phi\in\mathcal{D}(X)$; it will suffice to prove that $\phi u\in H^{s+\nu}(\mathbb{R}^n)$. Let $\psi_0,\psi_1,\dots,\psi_m$ (where the value of $m$ will be decided later) be test functions in $\mathcal{D}(\mathbb{R}^n)$ with supports as shown in Figure \ref{fig1}, i.e. such that:

\begin{figure*}
\centering
  \includegraphics[width=0.5\textwidth]{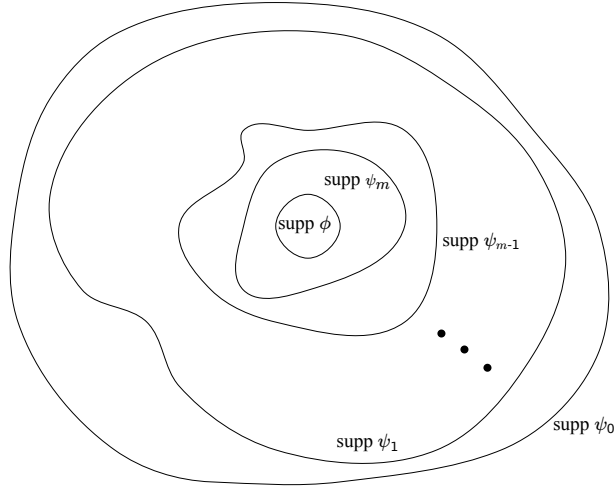}
\caption{The domains involved in the bootstrapping proof of Theorem \ref{FERT}}
\label{fig1}       
\end{figure*}

\begin{equation}
\begin{split}
\mathrm{supp}(\phi)\subset\mathrm{supp}(\psi_m)&,\;\;\psi_m=1\text{ on supp}(\phi); \\
\mathrm{supp}(\psi_i)\subset\mathrm{supp}(\psi_{i-1})&,\;\;\psi_{i-1}=1\text{ on supp}(\psi_i)\;\;\;\forall i.
\end{split}
\end{equation}
Now $\psi_0u$ is in $\mathcal{E}'(\mathbb{R}^n)$ and therefore in $H^t(\mathbb{R}^n)$ for some $t\in\mathbb{R}$. So

\begin{alignat*}{2}
P(D)(\psi_1u)&=\psi_1P(D)u+[P(D),\psi_1]u &&\text{   (where $[,]$ is a commutator)} \\
&= \psi_1f+[P(D),\psi_1](\psi_0u) &&\text{   (by (2))} \\
&= (\text{element of } H^s(\mathbb{R}^n))+(\text{element of } H^{t-\nu+1)}(\mathbb{R}^n)) &&\text{   (by Lemma \ref{SobolevConvolution})} \\
&\in H^{\min(s,t-\nu+1)}(\mathbb{R}^n) &&\text{   (since $a>b\Rightarrow H^a\subset H^b$)}.
\end{alignat*}

\noindent Now the first part of the proof shows that $\psi_1u\in H^{A_1}(\mathbb{R}^n)$ where $A_1:=\min(s,t-\nu+1)+\nu=\min(s+\nu,t+1)$.

By exactly the same argument, $P(D)(\psi_2u)=\psi_2f+[P(D),\psi_2](\psi_1u)$ and $\psi_2u\in H^{A_2}(\mathbb{R}^n)$ where $A_2:=\min(s+\nu,A_1+1)=\min(s+\nu,t+2)$.

Continuing in this manner eventually yields $\psi_mu\in H^{\min(s+\nu,t+m)}(\mathbb{R}^n)$. Now set the natural number $m$ to be $\lceil s+\nu-t\rceil+1$, so that $\psi_mu\in H^{s+\nu}(\mathbb{R}^n)$, which means $\phi u\in H^{s+\nu}(\mathbb{R}^n)$ as required, by (2).
\end{proof}

\section{Conclusions}
\label{sec:concl}
The elliptic regularity theorem is an important result in the theory of PDEs, and its fractional counterpart should have no less significance in the theory of fractional PDEs. Elliptic fractional PDEs have already been studied in papers such as \cite{bisci,chen,caffarelli,dipierro}, which present various methods for analysing the solutions of certain classes of elliptic fractional PDE. The current work fits in with such results by providing a quick way of establishing important regularity properties of linear elliptic fractional PDEs.

As example applications of our work, we consider the following two simple corollaries.

\begin{corollary}
\label{hypoelliptic}
Let $P(D)$ be a fractional linear partial differential operator of the form described above. If it is elliptic, then it is also hypoelliptic.
\end{corollary}

\begin{proof}
Recall the definition of hypoellipticity: a partial differential operator $\partial$ is hypoelliptic if whenever $\partial u$ is a smooth function, so also is $u$ on the same domain.

If $P(D)$ is elliptic, then using all notation as in Theorem \ref{FERT}, we must have $f\in C^{\infty}(X)\Rightarrow u\in C^{\infty}(X)$, i.e. $P(D)$ is also hypoelliptic.
\end{proof}

\begin{corollary}
\label{Laplacian}
Consider the operator $\widetilde{\Delta}_{\alpha}\coloneqq\sum_{i=1}^n\partial_i^{\alpha}$ with $0<\alpha<1$, a fractional generalisation of the Laplacian, and a function $u\in\mathcal{D}'(X)$ where $X$ is a domain in $\mathbb{R}^n$.

If $u$ is a solution to the fractional Laplace-type equation $\widetilde{\Delta}_{\alpha}u=0$, then it must necessarily be smooth. More generally, if $u$ is the solution to a fractional Poisson-type equation $\widetilde{\Delta}_{\alpha}v=f$ with forcing $f\in H^s_{loc}(X)$, then $u\in H^{s+\alpha}_{loc}(X)$.
\end{corollary}

\begin{proof}
The fractional operator $\sum_{i=1}^n\partial_i^{\alpha}$ is elliptic when $0<\alpha<1$, since then $\lambda^{\alpha}$ is in the right half complex plane for all $\lambda\in\mathbb{R}$. So Theorem \ref{FERT} applies and the results follow.
\end{proof}

The result proved herein is only one of many possible versions of a fractional elliptic regularity theorem.

For classical PDEs, there are far more elliptic regularity theorems than Theorem \ref{ERT}, which covers only linear partial differential operators whose coefficients are constants in $\mathbb{C}$. Other versions concern linear partial differential operators with non-constant coefficients, perhaps satisfying some $C^k$ or $L^p$ condition; the Sobolev conditions can also sometimes be replaced by $L^p$ conditions on the functions $f$ and $u$. See e.g. \cite[Chapter 6C]{folland1} and \cite[Chapter 6.3]{evans}. These other variants of the elliptic regularity theorem may well be extendable to fractional PDEs just as Theorem \ref{ERT} was.

Furthermore, there are more models of fractional calculus than just the Riemann--Liouville formula. Some of them cooperate with the Fourier transform almost as well as Riemann--Liouville differintegrals do, which was a necessary factor in our proofs here. Thus, with a little more work we may be able to prove results analogous to Theorem \ref{FERT} for fractional PDEs defined using other fractional models, which have different real-world applications from the Riemann--Liouville one.

\section*{Acknowledgements}
This work was supported by a research student grant from the Engineering and Physical Sciences Research Council, UK. The author wishes to thank Anthony Ashton and Dumitru Baleanu for helpful discussions and recommendations to the literature.

\end{document}